\documentclass[11pt]{article}

\usepackage[mathscr]{eucal}
\usepackage{color}
\usepackage{amsmath}
\usepackage{amssymb}
\usepackage{amsfonts}
\usepackage{amscd}
\usepackage{amsthm}
\usepackage{latexsym}
\usepackage{graphicx}
\usepackage{subfig}
\usepackage[numbers,sort&compress]{natbib} 
\def\be{\begin{equation}}
\def\ee{\end{equation}}
\def\bse{\begin{subequations}}
\def\ese{\end{subequations}}

\usepackage{latexsym}

\let\er\eqref

\let\be\beta

\newcommand{\R}{{\mathbb R}}

\newtheorem{theorem}{Theorem}
\newtheorem{lemma}[theorem]{Lemma}

\newtheorem{corollary}[theorem]{Corollary}

\def\bse{\begin{subequations}}
\def\ese{\end{subequations}}

\usepackage{geometry}

\title{Adaptive dynamics of nonlocal competition models with heterogeneous resources}

\author{ {\bf\large Shen Bian}\thanks{Corresponding author. Email address: bianshen66@163.com}  \\
\
\\
{\it\small Department of Mathematical Sciences, Beijing University of Chemical Technology} \\
{\it\small Beijing, 100029, P.R. CHINA }\\
\vspace{1.0mm} }

\date{}

\begin{document}
\let\cleardoublepage\clearpage

\maketitle

\begin{abstract}
We investigate the long-time behavior of phenotype-structured models describing  evolutionary dynamics of asexual populations, and analyze the joint effects of nonlocal interactions and spatial resource distributions on the global dynamics of the two species. In the first part, we consider an integro-differential system without diffusion terms, where phenotypic changes are absent and the spatial distribution of resources for one species is heterogeneous while that of the other is homogeneous. Using an entropy method to address nonlocal interactions and resource heterogeneity, we prove that the species subject to heterogeneous resources converges to a Dirac mass concentrated at the peak of the phenotypic fitness landscape, which establishes the selection of the best adapted trait. Numerical experiments further provide a sufficient criterion to identify the positions of fitness peaks. In the second part, we extend our study to a nonlocal reaction-diffusion system involving a linear diffusion operator, where heritable phenotypic changes occur and both resources are spatially heterogeneous. Through an appropriate transformation, we overcome the challenges induced by resource heterogeneity and prove that the long-time limits of the two species under different interspecific competitive coefficients are given by distinct steady states of the parabolic system, with concentrations at the maxima of their respective resource functions. Numerical results confirm the predictions and further reveal phenomena beyond the theoretical analysis.
\end{abstract}

\section{Introduction}

In the theory of adaptive dynamics which focuses on phenotypic evolution, populations are structured by a parameter describing a biological, physiological or ecological characteristic of the individuals. When this characteristic is inherent to the individuals, we refer to it as a phenotypic trait (denoted by a continuous real variable $x$ below) \cite{P15,P06}. The main ingredient in this theory is the selection principle which favors the population with best adapted trait in terms of resource utilization, and such a trait is called an evolutionary stable strategy (ESS). The origin of this denomination comes from evolution theory \cite{MP73}, no mutant with a different trait can invade a population with the trait corresponding to an ESS. In many scenarios of biological interest, one can assume $x \in \Omega,$ where $\Omega$ is a smoothly bounded domain in $\R^d$ or the entire space $\R^d$ for $d \ge 1.$ A classical example of such models is a two-species competition system 
 \begin{align}\label{twolocal}
  \left\{
      \begin{array}{ll}
       u_t=a_1 \Delta u+u \left( m_1(x)-u-bv  \right), & \text{in}~~\Omega \times \R^+, \\
       v_t=a_2 \Delta v+v \left(  m_2(x)-cu-v  \right), & \text{in}~~\Omega \times \R^+, \\
       \partial_{\nu} u=\partial_{\nu} v =0,  & \text{on}~~\partial \Omega \times \R^+, \\
      (u,v)(x,0)=(u_0,v_0)(x), & \text{in}~~\Omega,
      \end{array}\right.
\end{align}
where $\Omega$ is a smoothly bounded domain, $u(x,t)$ and $v(x,t)$ denote the population densities of two competing species. The functions $m_1(x)$ and $m_2(x)$ are either the carrying capacities or intrinsic growth rates, which model the spatial distribution of resources for species $u$ and $v$, respectively. The constants $b>0, c>0$ measure the inter-specific competitive abilities, while the intra-specific competition coefficients are normalized to one.

When the resource is spatially homogeneous, namely both $m_1(x)$ and $m_2(x)$ are positive constants, the global asymptotic dynamics of \er{twolocal} crucially depend on the reaction parameters and have been extensively studied, see \cite{B80,J10,LN96,H78,Z82} and \cite{G77,H78} (in the absence of diffusion terms). Precisely, the positive coexistence equilibrium is globally asymptotically stable if $c<\frac{m_2}{m_1}<\frac{1}{b}$, while competitive exclusion equilibrium is globally asymptotically stable if $\min\left(c,\frac{1}{b}\right)>\frac{m_2}{m_1}$ or $\max\left(c,\frac{1}{b}\right)<\frac{m_2}{m_1}$. Moreover, if $\frac{1}{b}<\frac{m_2}{m_1}<c,$ the two exclusion equilibria are locally stable (see \cite{LN96,IM98}).

When at least one of $m_1(x)$ and $m_2(x)$ is non-constant (the resource is spatially heterogeneous), a series of works provide a complete classification on the global asymptotic ability of the positive coexistence equilibrium and the competitive exclusion equilibrium of \er{twolocal}, see \cite{D98} for the case $b=c=1$ and further study in \cite{L06} for $b<1, c<1$. More complete parameter regimes for $b,c$ and carrying capacities $m_1(x),m_2(x)$ have been studied in a series of important works \cite{LN12,HN16I,HN16II,HN17,W25,GT23,GT24}.

We focus here on the case where, in the absence of phenotypic changes, the evolution of population densities is governed by an integro-differential system of the form
\begin{align}\label{uvnonlocal}
  \left\{
      \begin{array}{ll}
       u_t=u \left( d(x)-r_1(t)-b r_2(t)  \right), & (x,t) ~\text{in}~\R \times \R^+, \\
       v_t=v \left( \overline{m}-c r_1(t)-r_2(t) \right), & (x,t) ~\text{in}~\R \times \R^+, \\
      r_1(t)= \int_{\R}udx,~~r_2(t)=\int_{\R} v dx, & (x,t) ~\text{in}~\R \times \R^+, \\  
      (u,v)(x,0)=(u_0,v_0)(x) \ge 0, & x ~\text{in}~\R.
      \end{array}\right.
\end{align}
Here we take up the case where the spatial distribution of resources for one of the competing species is heterogeneous, while that of the other is uniform. The function $d(x)$ denotes the net per capita growth rate of individuals in the phenotypic state $x$, while $\overline{m}$ is the constant reproduction rate of $v.$ The terms $d(x)-r_1(t)-b r_2(t)$ and $\overline{m}-c r_1(t)-r_2(t)$ can be interpreted as the phenotypic fitness landscapes of species $u,v$, respectively, under the environment created by the total populations. 

In ecological and biological scenarios whereby the effects of heritable, spontaneous phenotypic changes need to be taken into account, a linear diffusion operator can be included in the integro-differential equations \er{uvnonlocal}. Here we take up the case where the spatial distribution of resources for two competing species is heterogeneous, which leads to non-local parabolic PDEs of the form
\begin{align}\label{uvnonlocaldiffusion}
  \left\{
      \begin{array}{ll}
       u_t=\Delta u+u \left( d(x)-r_1(t)-b r_2(t)  \right), & (x,t) ~\text{in}~\R \times \R^+, \\
       v_t=\Delta v+v \left( m(x)-c r_1(t)-r_2(t) \right), & (x,t) ~\text{in}~\R \times \R^+, \\
      r_1(t)= \int_{\R}udx,~~ r_2(t)= \int_{\R} v dx, & (x,t) ~\text{in}~\R \times \R^+, \\
      (u,v)(x,0)=(u_0,v_0)(x) \ge 0, & x ~\text{in}~\R,
      \end{array}\right.
\end{align}
where the diffusion terms model phenotypic changes and take into account mutations, the functions $d(x), m(x)$ represent the nonlocal carrying capacities or intrinsic growth rates, which reflect the environmental resources shared by species $u$ and $v$, respectively. The terms 
$d(x)-r_1(t)-b r_2(t)$ and $m(x)-c r_1(t)-r_2(t)$ describe the per-capita growth rates, capturing the combined effects of natural selection and competition.

In the framework of system \er{uvnonlocal} or \er{uvnonlocaldiffusion}, the maximum points of the functions $d(x), m(x)$ correspond to the fitness peaks (i.e. the peaks of the phenotypic fitness landscape of the populations). Moreover, the saturating terms $-r_1(t), -r_2(t)$ model the limitations on population growth imposed by carrying capacity constants. 

It is worthwhile to mention some aspects concerning the motivation of this paper. On one hand, nonlocal reaction terms can describe Darwinian evolution of a structured population density or the behavior of cancer cells with therapy (e.g. polychemotherapy and chemotherapy) \cite{L11,LL15}. On the other hand, whilst quite a number of works deal with models with local reaction terms in bounded domains, there is a paucity of literature concerning the case where nonlocal reaction terms are prescribed, with the exception of the asymptotic results for single species presented in \cite{BC12,BG16,CC07,DJM05,IM18,JS23,LC19,L20,LP21,L24,P06,PB08,PC18,PT18}. 

Nonlocal equations can create concentration effects leading to spikes. The purpose of this work is to illustrate the influences of nonlocal reaction terms and the inter-specific competition abilities on the dynamics of \er{uvnonlocal}, and to present the differences in the population dynamics between one species whose spatial resource distribution is heterogenous and the other whose spatial resource distribution is homogeneous. Additionally, we aim to explore the differences in long-time behavior between solutions to systems \er{uvnonlocal} and \er{uvnonlocaldiffusion}, and to investigate whether the presence of the diffusion terms changes the trend of natural selection. 

The remainder of the paper is organized as follows. In Section \ref{sec2}, we present the main assumptions and necessary preliminaries. Section \ref{sec3} is devoted to the asymptotic behavior of solutions to \er{uvnonlocal} (see Theorem \ref{th1} and Corollary \ref{coro1}), thereby establishing a selection principle. Section \ref{sec4} characterizes the long-time asymptotic behavior of solutions to \er{uvnonlocaldiffusion} (see Theorem \ref{nonlocaldiffusion}).

\section{Preliminaries} \label{sec2}

In this section, we compile some preparatory assumptions and lemmas that will be used in the proof of the long-time asymptotic behavior of solutions to \er{uvnonlocal} and \er{uvnonlocaldiffusion}.

Throughout this paper, we assume that the initial data satisfy
\begin{align}\label{assum1}
0<r_1(0)=\int_{\R} u_0(x)dx<\infty,~~0<r_2(0)=\int_{\R} v_0(x) dx<\infty.
\end{align}
The function $d(x)$ is assumed to be 
\begin{align}\label{assum2}
d(x) \in L^\infty(\R).
\end{align}

Let us next review the classical Lotka-Volterra system for the two competing species:
\begin{align}\label{ODEs}
\left\{
  \begin{array}{ll}
    Y_t=Y \left( \overline{d}-Y-bX \right), & t>0, \\
    X_t=X \left( \overline{m}-cY-X  \right), & t>0 \\
    Y(0)=Y_0>0, X(0)=X_0>0,
  \end{array}
\right.
\end{align}
where $Y$ and $X$ are the population densities of two competing species, $\overline{d}, \overline{m}$ are the intrinsic growth rates. The system \er{ODEs} has four equilibria: the coexistence equilibrium 
$$P_1:=(Y^\ast,X^\ast)=\left(\frac{\bar{d}-b \overline{m}}{1-bc}, \frac{\overline{m}-c \bar{d}}{1-bc}\right),$$ 
the competitive exclusion equilibria
$$P_2:=\left( \overline{d},0 \right),\quad P_3:=\left( 0,\overline{m} \right)$$ 
and the unstable equilibrium $(0,0).$ The dynamics of system \er{ODEs} have been extensively studied in the literature \cite{B80,G77,J10,LN96,H78,Z82} and can be summarized in the following lemma.
\begin{lemma}\label{odestability}
Under the assumptions $b>0, c>0$, the following statements hold for system \er{ODEs}.
\begin{itemize}
  \item[\textbf{(i)}] If $c<\frac{\overline{m}}{\overline{d}}$ and $b<\frac{\overline{d}}{\overline{m}}$, then the positive coexistence equilibrium $P_1 \left(\frac{\bar{d}-b \overline{m}}{1-bc}, \frac{\overline{m}-c \bar{d}}{1-bc}\right)$ is globally asymptotically stable.
  \item[\textbf{(ii)}] If $c>\frac{\overline{m}}{\overline{d}}$ and $b<\frac{\overline{d}}{\overline{m}}$, then the competitive exclusion equilibrium $P_2 \left( \overline{d},0 \right)$ is globally asymptotically stable.
  \item[\textbf{(iii)}] If $c<\frac{\overline{m}}{\overline{d}}$ and $b>\frac{\overline{d}}{\overline{m}}$, then the competitive exclusion equilibrium $P_3 \left( 0,\overline{m} \right)$ is globally asymptotically stable.
  \item[\textbf{(iv)}] If $bc=1$ and $b \overline{m}=\overline{d},$ then \er{ODEs} has a compact global attractor consisting of a continuum of steady states $\left(\eta \overline{d}, (1-\eta) \overline{d}/b \right)$ for $\eta \in [0,1].$
  \item[\textbf{(v)}] If $c>\frac{\overline{m}}{\overline{d}}$ and $b>\frac{\overline{d}}{\overline{m}}$, then $P_1 \left(\frac{\overline{d}-b \overline{m}}{1-bc}, \frac{\overline{m}-c \overline{d}}{1-bc}\right)$ is unstable and $P_2 \left( \overline{d},0 \right), P_3 \left( 0,\overline{m} \right)$ are locally stable.
\end{itemize}
\end{lemma}

In \cite{G77,H78}, Lemma \ref{odestability} (i)-(iii) are proved by applying a Lyapunov functional method. For the proof of Lemma \ref{odestability} (iv) one can refer to \cite{HN16I} for more details. In the case of (v), the equilibrium $P_1=(Y^\ast,X^\ast)=\left(\frac{\overline{d}-b \overline{m}}{1-bc}, \frac{\overline{m}-c \bar{d}}{1-bc}\right) $ is the saddle point, the basins of attraction for the two stable equilibria $P_2,P_3$ are separated by the stable manifold of the saddle point, denoted as $X=h(Y)$ such that 
\begin{align}
\left\{
  \begin{array}{ll}
    \text{if}~X(0)<h(Y(0)),~\text{then}~\displaystyle \lim_{t \to \infty} (Y(t),X(t))=(\overline{d},0), \\
\text{if}~X(0)>h(Y(0)),~\text{then}~\displaystyle \lim_{t \to \infty} (Y(t),X(t))=(0,\overline{m}). \\
  \end{array}
\right.
\end{align}
The properties of the separatrix $X=h(Y)$ have been analyzed in \cite{IM98,LN96} as presented in the following lemma. 
\begin{lemma}\label{saddle}
Suppose that $c>\frac{\overline{m}}{\overline{d}}$ and $b>\frac{\overline{d}}{\overline{m}}$. Then there exists a monotone function $h(Y)$ defined on $[0,\infty)$ whose graph is the separatrix of \er{ODEs}. Moreover, the function satisfies
\begin{itemize}
  \item[(i)] $h(0)=0; h(Y^\ast)=X^\ast; $
  \item[(ii)] $h'(Y)>0 ~~\text{in}~~ [0,\infty)$;
  \item[(iii)] $\displaystyle \lim_{Y \to \infty} h(Y) =\infty;$
  \item[(iv)] $\left\{
                \begin{array}{ll}
                  h''(Y)<0, & \overline{d}>\overline{m}, \\
                  h''(Y)=0, & \overline{d}=\overline{m} \\
                  h''(Y)>0, & \overline{d}<\overline{m};
                \end{array}
              \right.$
\item[(v)] $Y (\overline{d}-Y-bh(Y)) h'(Y)=h(Y)(\overline{m}-cY-h(Y))$.
\end{itemize}
\end{lemma}

In order to describe the stable manifold near the saddle point, we solve $h(Y)$ by combining linearization analysis with nonlinear corrections. Firstly, we set
\begin{align}
\vec{F}(\vec{x})=\left(
                   \begin{array}{c}
                     f(Y,X) \\
                     g(Y,X) \\
                   \end{array}
                 \right)=\begin{pmatrix}
       Y \left( \overline{d}-Y-bX \right) \\
       X \left( \overline{m}-cY-X  \right)
        \end{pmatrix}
\end{align}
and compute its Jacobian matrix at $(Y^\ast,X^\ast)$ to yield 
\begin{align}\label{linearmatrix}
A:=\frac{\partial \vec{F}(Y,X)}{\partial (Y,X)} \Bigg| _{(Y^\ast,X^\ast)}=\begin{pmatrix}
                         \bar{d}-2Y-bX & -b Y \\
                         -cX & \overline{m}-cY-2X \\
                       \end{pmatrix} \Bigg| _{(Y^\ast,X^\ast)}=\begin{pmatrix}
                                       -Y^\ast & -b Y^\ast \\
                                       -cX^\ast & -X^\ast \\
                                     \end{pmatrix}.
\end{align} 
The eigenvalues of $A$ are then given by
\begin{align}
\lambda_{1,2}=\frac{-(X^\ast+Y^\ast)\pm \sqrt{(X^\ast)^2+(Y^\ast)^2+(4bc-2)X^\ast Y^\ast}}{2}.
\end{align}
Subsequently, the eigenvector corresponding to the stable eigenvalue 
\begin{align}
\lambda_1=\frac{-(X^\ast+Y^\ast)- \sqrt{(X^\ast)^2+(Y^\ast)^2+(4bc-2)X^\ast Y^\ast}}{2}<0 
\end{align}
is $\vec{v}=\left( \frac{-b Y^\ast}{Y^\ast +\lambda_1},1 \right)^T$. Therefore, the linear approximation of the separatrix near the saddle point can be expressed as 
\begin{align}
X-X^\ast=k (Y-Y^\ast),
\end{align}
where 
\begin{align}\label{k}
k&=\frac{Y^\ast +\lambda_1}{-b Y^\ast}=\frac{-cX^\ast}{X^\ast+\lambda_1} \nonumber \\
&=\frac{Y^\ast-X^\ast- \sqrt{(X^\ast)^2+(Y^\ast)^2+(4bc-2)X^\ast Y^\ast} }{-2bY^\ast}.
\end{align}
Secondly, we look for the higher-order analytical approximation of $h(Y)$ subject to $h(Y^\ast)=X^\ast$ and $h'(Y^\ast)=k$. We expand $h(Y)$ as
\begin{align}
h(Y)=X^\ast+k(Y-Y^\ast)+a_2 (Y-Y^\ast)^2+O(|Y-Y^\ast|^3)
\end{align}
near $Y=Y^\ast.$ Substituting it into system \er{ODEs} along the stable manifold, we arrive at
\begin{align}\label{hyode}
\frac{d h(Y)}{dY} f(Y,h(Y))=g(Y,h(Y)).
\end{align}
On the other hand, $ f(Y,h(Y))$ and $g(Y,h(Y))$ are expanded as
\begin{align*}
 f(Y,h(Y))&=(-b k Y^\ast-Y^\ast) (Y-Y^\ast)+(-a_2 b Y^\ast-bk-1) (Y-Y^\ast)^2+O(|Y-Y^\ast|^3), \\
g(Y,h(Y))&=(-cX^\ast-kX^\ast) (Y-Y^\ast)+(-a_2 X^\ast-ck-k^2) (Y-Y^\ast)^2+O(|Y-Y^\ast|^3).
\end{align*}
Equating the coefficients termwise, we find 
\begin{align}
a_2=\frac{(c-1)k+(1-b)k^2}{-Y^\ast-X^\ast-3 \lambda_1}.
\end{align}
Finally, the stable manifold near the saddle point can be approximated as 
\begin{align}\label{secondappro}
X=X^\ast+k(Y-Y^\ast)+a_2 (Y-Y^\ast)^2.
\end{align}

The global stable manifold can be solved by numerical simulations. First of all, the initial data can be chosen near the saddle point along the direction of the stable eigenvector, namely
\begin{align}
(Y_0,X_0)^T=(Y^\ast,X^\ast)^T \pm \varepsilon \vec{v}.
\end{align}
Then, we numerically compute the backward-time solution of the initial value problem \er{ODEs} via the Runge-Kutta method. The resulting trajectory is the global stable manifold of the saddle point, which is the separatrix of the basins of attraction of the two stable equilibria $P_2$ and $P_3$. Moreover, the quadratic approximation \er{secondappro} is in good agreement with the global stable manifold near the saddle point (see the blue solid line and the red solid line plotted in Figure \ref{bistablecase}).


\section{Long-time asymptotic behavior of the Cauchy problem \er{uvnonlocal}} \label{sec3}

In this section, we investigate the asymptotic behavior of solutions to the cauchy problem \er{uvnonlocal}. Under assumptions \er{assum1}-\er{assum2}, there exists a unique non-negative mild solution $u(x,t) \in C\left( \R_+;L^1(\R) \right)$ of the Cauchy problem \er{uvnonlocal} \cite{DJM08,P06}. By solving the first equation of \er{uvnonlocal}, we directly obtain the semi-explicit formula for $u(x,t)$ 
\begin{align}\label{ux}
u(x,t)=u_0(x) exp \left(\int_0^t d(x)-r_1(s)-b r_2(s) ds\right)
\end{align}
which implies that the support of $u(x,t)$ remains unchanged over time. This property also holds for $v(x,t)$. We now define
\begin{align}\label{assum0}
\Omega:=supp u_0(x)=\{x \in \R ~ |~ u_0(x)>0 \}.
\end{align}
We assume that $d(x) \in C(\R)$ and there exist $d_m, d_M$ such that
\begin{align}\label{assum20}
\left\{
  \begin{array}{ll}
    0<d_m<d(x)<d_M, \\
    d_M:=\max_{x \in \Omega} d(x)~~\text{is attained for a single} ~~\overline{x} \in \Omega
  \end{array}
\right.
\end{align}
and there exists a pair $\left(r_1^\ast,r_2^\ast\right) \in \R^+ \times \R^+$ such that
\begin{align}\label{assum3}
\begin{array}{c}
\left\{
  \begin{array}{ll}
   d_M=r_1^\ast+b r_2^\ast, \\
\overline{m}=c r_1^\ast+r_2^\ast.
  \end{array}
\right. 
\end{array}
\end{align}
If $\Omega$ is unbounded, we also require that there exists $R>0$ such that as $(r_1,r_2) \to \left(r_1^\ast,r_2^\ast \right)$,
\begin{align}\label{assum4}
\beta_R:=\max_{|x| \ge R} \left[d(x)-r_1- b r_2 \right]<0.
\end{align}

\subsection{Asymptotic analysis}\label{sec31}

Motivated by Lemma \ref{odestability}, we shall primarily focus on the coexistence state of the two species for \er{uvnonlocal}. Setting
\begin{align}
r_1^\ast& :=\frac{d_M-b \overline{m}}{1-bc}, \\
r_2^\ast &:=\overline{m}-cr_1^\ast =\frac{\overline{m}-c d_M}{1-bc},
\end{align}
we will show that the limit measures of $u(x,t)$ and $v(x,t)$ are uniquely determined by the initial conditions, the function $d(x)$ and the constant $\overline{m}$, which entails the convergence of $u(\cdot,t)$ to the best adapted trait (the peak of the fitness) and of $v(\cdot,t)$ to a multiple of the initial condition as $t \to \infty.$ These results are stated in the following theorem.

\begin{theorem}\label{th1}
Under assumptions \er{assum1}, \er{assum20}-\er{assum4}. If $b>0, c>0$ and
\begin{align}\label{bcassum}
b<\frac{d_M}{\overline{m}}, ~~c<\frac{\overline{m}}{d_M}
\end{align}
hold, then the solutions to \er{uvnonlocal} will coexist and satisfy
\begin{align}
&r_1(t) \to r_1^\ast, \quad u(x,t) \rightharpoonup r_1^\ast \delta(x-\overline{x}),~~ \text{as}~~ t \to \infty, \\
&r_2(t) \to r_2^\ast, \quad v(x,t) \to \frac{r_2^\ast}{r_2(0)} v_0(x)~~\text{as}~~t \to \infty.
\end{align}
\end{theorem}
\begin{proof}
The proof is divided into several steps. We first provide a priori estimates for $r_1(t)$ and $r_2(t)$. Subsequently, a Lyapunov functional is constructed to establish the existence of the long-time limits for $r_1(t)$ and $r_2(t)$. Finally, the weak limits of $u(x,t)$ and $v(x,t)$ are rigorously identified, respectively.

{\it\textbf{Step 1}} (A priori estimates for $r_1(t)$ and $r_2(t)$) \quad From \er{ux}, we conclude that the support of $u(x,\cdot)$ remains unchanged over time, namely
\begin{align}\label{support}
\Omega=\{x \in \R ~\big |~ u(x,t)>0 \}=\{x \in \R ~\big |~ u_0(x)>0 \},~~\forall ~t > 0.
\end{align}
Under assumption \er{assum20}, we integrate system \er{uvnonlocal} over $\R$ to obtain
\begin{align}\label{uvnonlocal1}
\left\{
  \begin{array}{ll}
    \frac{d}{dt}r_1(t) \le r_1(t) \left( d_M-r_1(t)-b r_2(t)  \right), \\[1mm]
  \frac{d}{dt}r_2(t)=r_2(t) \left( \overline{m}-c r_1(t)-r_2(t)  \right), \\[1mm]
  0<r_1(0)<\infty, 0<r_2(0)<\infty.
  \end{array}
\right.
\end{align}
Then \er{bcassum} allows us to apply phase plane analysis to the ODE system solved by $r_1(t)$ and $r_2(t)$. We thereby obtain that for all times,
\begin{align}\label{r1r2bdd}
\left\{
  \begin{array}{ll}
  0 < r_1(t) \le \max(d_M,r_1(0)), \\
0 < r_2(t) \le \max(\overline{m},r_2(0)).
  \end{array}
\right.
\end{align}

{\it\textbf{Step 2}} (A Lyapunov functional) \quad In this step, a Lyapunov functional is constructed to establish the convergence of $r_1(t)$. We consider a function $P(r_1,r_2)$ satisfying
\begin{align}
P(r_1,r_2):=\frac{cr_1^2+2bcr_1r_2+b(r_2-\overline{m})^2}{2cr_1},
\end{align}
then it follows that
\begin{align}
P+r_1 P_{r_1}&=r_1+br_2, \\
r_1 P_{r_2}&=\frac{b}{c} \left( cr_1+r_2-\overline{m} \right).
\end{align}
Subsequently, direct calculations yield that
\begin{align*}
&\frac{d}{dt} \int_{\R} \left( d(x)-P(r_1,r_2)  \right) udx \\
=&\int_{\R} \left( d(x)-P-r_1 P_{r_1}  \right)\left( d(x)-r_1-br_2  \right) u  dx-\int_{\R} P_{r_2} r_1 v \left( \overline{m}-c r_1-r_2 \right) dx \\
=& \int_{\R} \left( d(x)-r_1-br_2  \right)^2 u  dx+\left( cr_1+r_2-\overline{m} \right) P_{r_2} r_1 r_2 \\
=& \int_{\R} \left( d(x)-r_1-br_2  \right)^2 u  dx+ \frac{b}{c}\left( cr_1+r_2-\overline{m} \right)^2 r_2 \ge 0.
\end{align*}
Define
\begin{align*}
I_1&:=\int_{\R} \left( d(x)-r_1-br_2  \right)^2 u  dx, \\
I_2&:=r_2 \left( cr_1+r_2-\overline{m} \right)^2.
\end{align*}
As a consequence, the bounded quantity $\int_{\R} \left( d(x)-P(r_1,r_2)  \right) udx$ is increasing in time and thus converges as $t \to \infty.$ Hence we have
\begin{align}\label{limitL}
\int_{\R} \left( d(x)-P(r_1,r_2)  \right) udx \xrightarrow[t \to \infty]{} L \in \R
\end{align}
and
\begin{align}\label{condition1}
\int_0^\infty I_1+I_2 dt <\infty.
\end{align}

{\it\textbf{Step 3}} (Existence of the limit for $r_1(t)$) \quad In this step, we shall derive the limits of $r_1(t)$ and $r_2(t)$. To this end, we compute
\begin{align*}
&\frac{d}{dt} (I_1+I_2) \\
=& \frac{d}{dt} \int_{\R} \left( d(x)-r_1-br_2  \right)^2 u  dx+ \frac{d}{dt} \left[r_2 \left( cr_1+r_2-\overline{m} \right)^2 \right] \\
=& \int_{\R} \left( d(x)-r_1-br_2  \right)^2 \left( d(x)-r_1-br_2  \right) u  dx-2 \int_{\R} \left( d(x)-r_1-br_2  \right) u  dx  \int_{\R} \left( d(x)-r_1-br_2  \right) u  dx \\
&+2b r_2 \left( cr_1+r_2-\overline{m} \right)  \int_{\R} \left( d(x)-r_1-br_2  \right) u  dx+r_2 \left( \overline{m}-cr_1-r_2 \right) \left( cr_1+r_2-\overline{m} \right)^2 \\
&+2r_2^2 \left( cr_1+r_2-\overline{m} \right)^2+ 2c r_2 \left( cr_1+r_2-\overline{m} \right)  \int_{\R} \left( d(x)-r_1-br_2  \right) u  dx \\
=&\int_{\R} \left( d(x)-r_1-br_2  \right)^2 \left( d(x)-r_1-br_2  \right) u  dx-2 \int_{\R} \left( d(x)-r_1-br_2  \right) u  dx  \int_{\R} \left( d(x)-r_1-br_2  \right) u  dx \\
&+ r_2 \left( \overline{m}-cr_1+r_2 \right) \left( cr_1+r_2-\overline{m} \right)^2+2(b+c) r_2 \left( cr_1+r_2-\overline{m} \right)  \int_{\R} \left( d(x)-r_1-br_2  \right) u  dx.
\end{align*}
Therefore, applying Cauchy-Schwarz inequality for the second term and the fourth term in the right-hand side yields
\begin{align*}
\left(\int_{\R} \left( d(x)-r_1-br_2  \right) u  dx\right)^2
=& \left(\int_{\R} \left( d(x)-r_1-br_2  \right) \sqrt{u} \sqrt{u}  dx\right)^2 \\
 \le & r_1(t) \int_{\R} \left( d(x)-r_1-br_2  \right)^2 udx.
\end{align*}
By the boundedness of $r_1,r_2,d(x),\overline{m}$, we deduce that
\begin{align}
\Big| \frac{d}{dt} (I_1+I_2)  \Big| \le C (I_1+I_2).
\end{align}
Thanks to the integrability in \er{condition1}, this gives the following estimate
\begin{align}\label{condition2}
\int_0^{\infty} \Bigg| \frac{d}{dt} (I_1+I_2) \Bigg| dt<\infty.
\end{align}
Hence, taking \er{condition1} and \er{condition2} together gives rise to
\begin{align}\label{limitI1I2}
\displaystyle \lim_{t \to \infty} (I_1+I_2)=0
\end{align}
and thus we have
\begin{align}
&\displaystyle \lim_{t \to \infty} \textstyle \int_{\R} \left( d(x)-r_1-br_2  \right)^2 u  dx=0, \\
&\displaystyle \lim_{t \to \infty} \textstyle \int_{\R} v \left( cr_1+r_2-\overline{m} \right)^2 dx=0.
\end{align}
Again, the use of Cauchy-Schwarz inequality results in
\begin{align}
& \textstyle \int_{\R} \big| d(x)-r_1-br_2 \big| u dx  \xrightarrow[t \to \infty]{} 0, \label{eq1} \\[1mm]
&r_2 \big| cr_1+r_2-\overline{m} \big| \xrightarrow[t \to \infty]{} 0. \label{eq2}
\end{align}
Combining this with \er{limitL}, we further arrive at
\begin{align}
\int_{\R} \left(d(x)-P(r_1,r_2) \right) u dx=\int_{\R} [d(x)-r_1-br_2]+[r_1+br_2-P(r_1,r_2)] u dx \xrightarrow[t \to \infty]{} L,
\end{align}
and thus $[r_1+br_2-P(r_1,r_2)] r_1 $ has a limit. Recalling
\begin{align}
r_2(cr_1+r_2-\overline{m}) \to 0 ~\text{as}~ t \to \infty, 
\end{align}
we finally conclude from \er{r1r2bdd} that
\begin{align}
\displaystyle \lim_{t \to \infty} r_1(t) &=\overline{r}_1, \\
\displaystyle \lim_{t \to \infty} r_2(t) &=\overline{m}-c \overline{r}_1=\overline{r}_2.
\end{align}

{\it\textbf{Step 4}} (Identifying $\overline{r}_1=r_1^\ast, \overline{r}_2=r_2^\ast$) \quad This step aims to identify that the limit of $r_1(t)$ is $r_1^\ast=\frac{d_M-b \overline{m}}{1-bc}$, which directly follows that the limit of $r_2(t)$ is $r_2^\ast=\overline{m}-c r_1^\ast.$ The proof proceeds by contradiction. Given $\overline{r}_2=\overline{m}-c \overline{r}_1,$ if $\overline{r}_1>\frac{d_M-b \overline{m}}{1-bc}$, then there exists $t_0$ such that for $t>t_0$ large enough, we obtain
\begin{align}\label{626}
\max_{x \in \Omega} \left( d(x)-r_1(t)-b r_2(t) \right)<\max_{x \in \Omega} \left( d(x)-r_1^\ast-b r_2^\ast \right)=0,
\end{align}
and \er{uvnonlocal1} implies extinction of $r_1(t)$ which contradicts with the assumption. Conversely, if $\overline{r}_1<\frac{d_M-b \overline{m}}{1-bc}$, then for time large enough, we find
\begin{align}
\max_{x \in \Omega} \left( d(x)-r_1(t)-b r_2(t) \right) >0,
\end{align}
and $r_1(t)$ admits exponential growth for those $x$s where $d(x)-r_1(t)-b r_2(t)>0$, which is again a contradiction. Collecting both cases we conclude that as $t \to \infty,$
\begin{align}\label{240520}
&r_1(t) \to r_1^\ast=\frac{d_M-b \overline{m}}{1-bc}, \\
&r_2(t) \to r_2^\ast=\overline{m}-cr_1^\ast.
\end{align}

{\it\textbf{Step 5}} (The limit for $u(x,t)$) \quad Recalling the fact \er{support}, if $\Omega$ is bounded, it directly follows that $u(x,t)$ converges weakly to a measure $u^\ast(x)$ endowed with  $r_1^\ast=\int_{\R} u^\ast(x) dx$. If $\Omega$ is unbounded, \er{assum4} ensures that for $t$ large enough,
\begin{align}
\frac{d}{dt} \int_{|x|>R} u(x,t)dx \le \beta_R \int_{|x|>R} u(x,t) dx,
\end{align}
and thereby $\sup_{t>0}\int_{|x|>R} u(x,t)dx \to 0$ for $R$ large enough. This indicates that the family $(u(x,t))_{t>0}$ is compact in the weak sense of measures. Hence there are subsequences $u(x,t_k)$ that converge weakly to measures $u^\ast(x)$ and $r_1^\ast=\int_{\R} u^\ast(x) dx$ as $k \to \infty.$

{\it\textbf{Step 6}} (Identification of the limits for $u(x,t)$ and $v(x,t)$) \quad Now we are ready to determine the limit for $u(x,t)$. It follows from \er{eq1} and \er{240520} that $u^\ast(x)$ concentrates on the set of points $x$s such that $d(x)-r_1^\ast-b r_2^\ast=0$. Furthermore, with the help of \er{assum3} and \er{assum4}, the point of concentration is unique and thus
\begin{align}\label{eq1018}
u^\ast(x)=r_1^\ast \delta(x-\overline{x}).
\end{align}
Therefore, the family $u(x,t)$ converges uniformly. Indeed, for $x \neq \overline{x}$, we have $d(x)<d(\overline{x})$ and
\[
u(x,t)=u_0(x)~e^{ \int_0^t \left(d(x)-r_1(s)-b r_2(s) \right) ds }\to 0,~~\text{as}~~ t \to \infty.
\]
In addition, integrating the second equation of \er{uvnonlocal} leads to
\begin{align}
r_2(t)=r_2(0) e^{\int_0^t \overline{m}-cr_1(s)-r_2(s) ds}.
\end{align}
Combining with the fact that
\begin{align}
v(x,t)=v_0(x) e^{\int_0^t \overline{m}-cr_1(s)-r_2(s) ds},
\end{align}
we thus end up with
\begin{align}
v(x,t)=v_0(x) \frac{r_2(t)}{r_2(0)}.
\end{align}
Therefore, we arrive at
\begin{align}
v(x,t) \to \frac{r_2^\ast}{r_2(0)} v_0(x)~~\text{as}~~t \to \infty.
\end{align}
Hence the desired results are proved.~~$\Box$
\end{proof}

As we have just seen in the proof of coexistence above, we deduce that as time goes to infinity, 
\begin{align}
\textstyle \int_{\R} | d(x)-r_1(t)-br_2(t) | u(x,t) dx  \xrightarrow[t \to \infty]{} 0
\end{align}
and
\begin{align}
r_2(t) \big| cr_1(t)+r_2(t)-\overline{m} \big| \xrightarrow[t \to \infty]{} 0.
\end{align}
By virtue of Lemma \ref{odestability} and Lemma \ref{saddle}, it follows from \eqref{uvnonlocal1} that we can then proceed analogously to the proof of Theorem \ref{th1} to derive the following results.
\begin{corollary}\label{coro1}
Under assumptions \er{assum1}, \er{assum20}-\er{assum4}, we assume that $b>0, c>0$, and the function $h(x)$ satisfies
\begin{align*}
\left\{
  \begin{array}{ll}
h(0)=0,~~h(r_1^\ast)=r_2^\ast, \\
x (d_M-x-bh(x)) h'(x)=h(x)(\overline{m}-cx-h(x)).
  \end{array}
\right.
\end{align*}
Then the following statements hold for system \er{uvnonlocal}.
\begin{itemize}
  \item[\textbf{(i)}] If either $c>\frac{\overline{m}}{d_M}, b<\frac{d_M}{\overline{m}}$ or $c>\frac{\overline{m}}{d_M}, b>\frac{d_M}{\overline{m}}$ with $r_2(0)<h(r_1(0))$, then the solutions to \er{uvnonlocal} satisfy
\begin{align}
u(x,t) \rightharpoonup d_M \delta(x-\overline{x}),~~
v(x,t) \to 0,~~\text{as}~~t \to \infty.
\end{align}
  \item[\textbf{(ii)}] If either $c<\frac{\overline{m}}{d_M}, b>\frac{d_M}{\overline{m}}$ or $c>\frac{\overline{m}}{d_M}, b>\frac{d_M}{\overline{m}}$ with $r_2(0)>h(r_1(0))$, then the solutions to \er{uvnonlocal} satisfy
\begin{align}
u(x,t) \rightharpoonup 0,~~
v(x,t) \to  \frac{\overline{m}}{r_2(0)} v_0(x),~~\text{as}~~t \to \infty.
\end{align}
  \item[\textbf{(iii)}] If $bc=1$ and $b \overline{m}=d_M,$ then as time goes to infinity, the solutions to \er{uvnonlocal} satisfy
\begin{align}
u(x,t) \rightharpoonup \bar{r}_1 \delta(x-\overline{x}),~~
v(x,t) \to  \frac{\bar{r}_2}{r_2(0)} v_0(x),~~\text{as}~~t \to \infty.
\end{align}
where $\bar{r}_1,\bar{r}_2$ lies on the line $\bar{r}_1+b \bar{r}_2=d_M$.
\end{itemize}
\end{corollary}

\section{Long-time behavior of the reaction-diffusion equation} \label{sec4}

In this section, we investigate the long-time asymptotic behavior of solutions to the reaction-diffusion system \er{uvnonlocaldiffusion}, with a focus on classifying the parameters $b,c$ to illustrate how diffusion terms (mutations) influence natural selection. Then we provide numerical examples to validate the theoretical predictions.  

\subsection{Asymptotic analysis}\label{sec41}

Under assumptions \er{assum1} and \er{assum2}, with the additional requirement that $m(x) \in L^\infty(\R)$, the Cauchy problem \er{uvnonlocaldiffusion} admits a unique non-negative solution $u(x,t) \in C( \R^+; L^1(\R) ) \cap C^1 ((0,\infty),C^{2,\alpha}(\R)) (0<\alpha<1)$ as established in \cite{P15,L20,PB08}. We now define the stationary equations 
\begin{align}\label{nonlocalsteady}
\left\{
  \begin{array}{ll}
   -\overline{u}''(x)=\overline{u}(x) \left( d(x)-\overline{r}_1-b \overline{r}_2 \right), & x \in \R, \\
   -\overline{v}''(x)=\overline{v}(x) \left( m(x) -c \overline{r}_1- \overline{r}_2 \right), & x\in \R, \\
   \overline{r}_1=\int_{\R} \overline{u}dx,\quad \overline{r}_2=\int_{\R} \overline{v}dx, \\
   \overline{u}(\pm \infty)=\overline{v}(\pm \infty)=0.
  \end{array}
\right.
\end{align}
Then the following results characterize the asymptotic behavior of dynamical solutions to \er{uvnonlocaldiffusion}.

\begin{theorem}\label{nonlocaldiffusion}
Let $b>0, c>0.$ The function $h(x)$ satisfy 
\begin{align*}
\left\{
  \begin{array}{ll}
    h(0)=0,~~h(\overline{r}_1)=\overline{r}_2, \\
x (\overline{r}_1+b \overline{r}_2-x-bh(x)) h'(x)=h(x)(c \overline{r}_1+ \overline{r}_2-cx-h(x)).
  \end{array}
\right.
\end{align*}
$K_1, K_2$ are defined as 
\begin{align}
K_1 \int_{\R} \overline{u}^2(x) dx=\int_{\R} \overline{u}(x) u_0(x) dx, \\
K_2 \int_{\R} \overline{v}^2(x) dx=\int_{\R} \overline{v}(x) v_0(x) dx. 
\end{align}
Under assumption \er{assum1}, we assume that $d(x) \in L^\infty(\R), m(x) \in L^\infty(\R)$. Then the following statements hold for system \er{uvnonlocaldiffusion}.
\begin{itemize}
  \item[\textbf{(i)}] If $bc<1$, then
  \begin{align}
     u(x,t) \xrightarrow[t \to \infty]{} \overline{u}(x) \quad \text{and}\quad v(x,t) \xrightarrow[t \to \infty]{} \overline{v}(x)\quad \text{in}\quad L^\infty(\R).
  \end{align}  
  \item[\textbf{(ii)}] For $bc>1$. If 
  \begin{align}\label{diffusionK2xiaoyuK1}
  K_2 \overline{r}_2<h (K_1 \overline{r}_1),
  \end{align}
  then
  \begin{align}
u(x,t) \xrightarrow[t \to \infty]{} \frac{\overline{r}_1+b \overline{r}_2}{\overline{r}_1} \overline{u}(x), \quad v(x,t) \xrightarrow[t \to \infty]{}  0~~\text{in}~~L^\infty(\R).
\end{align}  
If 
  \begin{align}\label{diffusionK2dayuK1}
  K_2 \overline{r}_2>h (K_1 \overline{r}_1),
  \end{align}
  then 
  \begin{align}
u(x,t) \xrightarrow[t \to \infty]{} 0,\quad v(x,t) \xrightarrow[t \to \infty]{}  \frac{c \overline{r}_1+ \overline{r}_2}{\overline{r}_2} \overline{v}(x)~~\text{in}~~L^\infty(\R).
\end{align}  
\end{itemize}
\end{theorem}
\begin{proof}
We first consider non-negative solutions $\tilde{u}(x,t), \tilde{v}(x,t)$ to the linear system
\begin{align}\label{star1}
\left\{
  \begin{array}{ll}
 \partial_t \tilde{u} = \Delta \tilde{u} + \tilde{u}\left(d(x) - \overline{r}_1 - b \overline{r}_2\right), & x \in \mathbb{R},\ t \geq 0, \\
\partial_t \tilde{v} = \Delta \tilde{v} + \tilde{v}\left(\overline{m} - c \overline{r}_1 - \overline{r}_2\right), & x \in \mathbb{R},\ t \geq 0, \\
\tilde{u}(x,0) = u_0(x), \quad \tilde{v}(x,0) = v_0(x), \\
\tilde{u}(\pm\infty,t) = \tilde{v}(\pm\infty,t) = 0, & t \geq 0.
  \end{array}
\right.
\end{align}
This is a standard parabolic system with zero as the principle eigenvalue of the steady equations. We know from the general theory of dominant eigenvectors of positive operators (Krein-Rutman theorem and relative entropy \cite{L20,P15}) that 
\begin{align}
\tilde{u}(x,t) \xrightarrow[t \to \infty]{} K_1 \overline{u}(x)~~\text{in}~~L^\infty(\R)~~\text{for some}~~K_1>0, \label{star2} \\
\tilde{v}(x,t) \xrightarrow[t \to \infty]{} K_2 \overline{v}(x)~~\text{in}~~L^\infty(\R)~~\text{for some}~~K_2>0, \label{star3}
\end{align}
Furthermore, multiplying the first equation of \er{star1} by $\overline{u}(x)$ we find
\begin{align*}
\frac{d}{dt} \int_{\R} \overline{u}(x) \tilde{u}(x,t) dx &=\overline{u}(x) \left(\Delta \tilde{u}+\tilde{u}\left(d(x)-\overline{r}_1-b \overline{r}_2  \right) \right) \\
&=\tilde{u} \left( \Delta \overline{u}(x) +\overline{u}(x) \left(d(x)-\overline{r}_1-b \overline{r}_2  \right) \right) \\
&=0,
\end{align*}
which yields $\int_{\R} \overline{u}(x) \tilde{u}(x,t) dx=\int_{\R} \overline{u}(x) u_0(x) dx$ for any $t>0.$ Letting $t \to \infty$ we arrive at
\begin{align}\label{K1}
K_1 \int_{\R} \overline{u}^2(x) dx=\int_{\R} \overline{u}(x) u_0(x) dx.
\end{align}
Similarly, multiplying $\overline{v}(x)$ to the second equation of \er{star1} gives
\begin{align}\label{K2}
K_2 \int_{\R} \overline{v}^2(x) dx=\int_{\R} \overline{v}(x) v_0(x) dx.  
\end{align}

On the other hand, we look for solutions to \er{uvnonlocaldiffusion} of the form
\begin{align}
u(x,t)=w(t) \tilde{u}(x,t),\quad v(x,t)=z(t) \tilde{v}(x,t)
\end{align}
with $w(t) \ge 0, z(t) \ge 0.$ A direct computation shows
\begin{align*}
&\frac{\partial}{\partial t} u(x,t)-\Delta u-u \left( d(x)-r_1(t)-b r_2(t)  \right) \\
= & w'(t) \tilde{u}(x,t)+w(t) \tilde{u}(x,t) \left( r_1(t)-\overline{r}_1+b r_2(t)-b \overline{r}_2 \right)=0
\end{align*}
and
\begin{align*}
&\frac{\partial}{\partial t} v(x,t)-\Delta v-v \left( m(x)-c r_1(t)- r_2(t)  \right) \\
= & z'(t) \tilde{v}(x,t)+z(t) \tilde{v}(x,t) \left( c r_1(t)-c \overline{r}_1+ r_2(t)- \overline{r}_2 \right)=0,
\end{align*}
which allow us to find $w(t), z(t)$ by the equations
\begin{align}
\left\{
  \begin{array}{ll}
    w'(t)+w(t)\left( r_1(t)-\overline{r}_1+b r_2(t)-b \overline{r}_2 \right)=0, \\
    z'(t)+z(t)\left( c r_1(t)-c \overline{r}_1+ r_2(t)- \overline{r}_2 \right)=0, \\
w(0)=1,\quad z(0)=1.
  \end{array}
\right.
\end{align}
Since 
\begin{align}
r_1(t)=\int_{\R} u(x,t)dx=w(t) \int_{\R} \tilde{u}(x,t)dx=w(t) \lambda_1(t)
\end{align}
with
\begin{align}
\lambda_1(t)=\int_{\R} \tilde{u}(x,t) dx \xrightarrow[t \to \infty]{} K_1 \overline{r}_1
\end{align}
and 
\begin{align}
r_2(t)=\int_{\R} v(x,t)dx=z(t) \int_{\R} \tilde{v}(x,t)dx=z(t) \lambda_2(t)
\end{align}
with
\begin{align}
\lambda_2(t)=\int_{\R} \tilde{v}(x,t) dx \xrightarrow[t \to \infty]{} K_2 \overline{r}_2,
\end{align}
we reformulate the system as:
\begin{align}\label{wz}
\left\{
  \begin{array}{ll}
    w'(t)+w(t)\left( w(t)\lambda_1(t)+b z(t) \lambda_2(t)-\overline{r}_1-b \overline{r}_2 \right)=0, \\
z'(t)+z(t)\left(c w(t)\lambda_1(t)+ z(t) \lambda_2(t)-c \overline{r}_1- \overline{r}_2 \right)=0, \\
w(0)=z(0)=1.
  \end{array}
\right.
\end{align}
Equivalently, the ODEs \er{wz} can be rewritten as
\begin{align}
\left\{
  \begin{array}{ll}
    K_1 \overline{r}_1w'(t)+ K_1 \overline{r}_1 w(t)\left( w(t) K_1 \overline{r}_1+b z(t) K_2 \overline{r}_2-\overline{r}_1-b \overline{r}_2 \right)=-K_1 \overline{r}_1 w(t)f_1(t), \\
    K_2 \overline{r}_2z'(t)+K_2 \overline{r}_2z(t)\left(c w(t)K_1 \overline{r}_1+ z(t)K_2 \overline{r}_2 -c \overline{r}_1- \overline{r}_2 \right)=-K_2 \overline{r}_2 z(t)f_2(t), \\
w(0)=z(0)=1
  \end{array}
\right.
\end{align}
with the functions
\begin{align*}
f_1(t)&=w(t) \left( \lambda_1(t)-K_1 \overline{r}_1  \right)+b z(t) \left( \lambda_2(t)-K_2 \overline{r}_2 \right), \\
f_2(t)&=c w(t) \left( \lambda_1(t)-K_1 \overline{r}_1  \right)+ z(t) \left( \lambda_2(t)-K_2 \overline{r}_2 \right)
\end{align*}
being such that $f_1(t) \to 0, f_2(t) \to 0$ as $t \to \infty.$ By applying case $(i)$ of Lemma \ref{odestability}, we conclude that for $bc<1,$ the inequalities 
$$b(c \overline{r}_1+ \overline{r}_2)<\overline{r}_1+b \overline{r}_2 ~~\text{and}~~ c(\overline{r}_1+b \overline{r}_2)<c \overline{r}_1+ \overline{r}_2 $$ 
hold, which imply
\begin{align}
w(t) \xrightarrow[t \to \infty]{} \frac{1}{K_1}, \quad z(t) \xrightarrow[t \to \infty]{} \frac{1}{K_2}.
\end{align}
Here $K_1, K_2$ are defined according to \er{K1} and \er{K2}, respectively. Recalling that $u(x,t)=w(t)\tilde{u}(x,t)$, we obtain
\begin{align}
&u(x,t) \xrightarrow[t \to \infty]{} \overline{u}(x) ~~\text{in}~~L^\infty(\R), \\
&v(x,t) \xrightarrow[t \to \infty]{} \overline{v}(x)~~\text{in}~~L^\infty(\R)
\end{align}
after using \er{star2} and \er{star3}. 

Now we turn to the case $bc>1$, which implies
$$b(c \overline{r}_1+ \overline{r}_2)>\overline{r}_1+b \overline{r}_2~~\text{and}~~ c(\overline{r}_1+b \overline{r}_2)>c \overline{r}_1+ \overline{r}_2.$$ 
By case $(v)$ of Lemma 
\ref{odestability}, we deduce that if $K_2 \overline{r}_2 z(0)>h(K_1 \overline{r}_1 w(0))$, then
\begin{align}
w(t) \xrightarrow[t \to \infty]{} 0, \quad z(t) \xrightarrow[t \to \infty]{} \frac{c \overline{r}_1+ \overline{r}_2}{K_2 \overline{r}_2}
\end{align}
and consequently
\begin{align}
u(x,t) \xrightarrow[t \to \infty]{} 0,\quad v(x,t) \xrightarrow[t \to \infty]{}  \frac{c \overline{r}_1+ \overline{r}_2}{\overline{r}_2} \overline{v}(x)~~\text{in}~~L^\infty(\R).
\end{align}
While if $K_2 \overline{r}_2 z(0)<h(K_1 \overline{r}_1 w(0))$, then
\begin{align}
w(t) \xrightarrow[t \to \infty]{} \frac{\overline{r}_1+b \overline{r}_2}{K_1 \overline{r}_1}, \quad z(t) \xrightarrow[t \to \infty]{} 0,
\end{align}
which leads to
\begin{align}
u(x,t) \xrightarrow[t \to \infty]{} \frac{\overline{r}_1+b \overline{r}_2}{\overline{r}_1} \overline{u}(x), \quad v(x,t) \xrightarrow[t \to \infty]{}  0~~\text{in}~~L^\infty(\R)
\end{align}
after applying \er{star2} and \er{star3}. This is the anticipated result.
\end{proof}




\begin{thebibliography}{99}

\bibitem{BC12} E. Bouin, V. Calvez, N. Meunier, S. Mirrahimi, B. Perthame, G. Raoul and R. Voituriez, Invasion fronts with variable motility: phenotype selection, spatial sorting and wave acceleration, {\it Comptes Rendus Math\'{e}matique}, {\bf 350}(15-16)  (2012) 761-766. 

\bibitem{B80} P.N. Brown, Decay to uniform states in ecological interactions, {\it SIAM Journal on Applied Mathematics}, {\bf 38} (1980) 22-37.

\bibitem{BG16} J.-E. Busse, P. Gwiazda and A.M. Czochra, Mass concentration in a nonlocal model of clonal selection, {\it Journal of Mathematical Biology}, {\bf 73}(4)  (2016) 1001-1033. 
    
\bibitem{CC07} A. Calsina, S. Cuadrado, Asymptotic stability of equilibria of selection-mutation equations, {\it Journal of Mathematical Biology}, {\bf 54}(4)  (2007) 489-511. 

\bibitem{DJM08} L. Desvillettes, P.E. Jabin, S. Mischler and G. Raoul, On selection dynamics for continuous structured populations, {\it Communications in Mathematical Sciences}, {\bf 6}(3) (2008) 729-747.

\bibitem{DJM05} O. Diekmann, P.E. Jabin, S. Mischler and B. Perthame, The dynamics of adaptation: an illuminating example and a Hamilton-Jacobi approach, {\it Theoretical Population Biology}, {\bf 67}(4) (2005) 257-271.
    
\bibitem{D98} J. Dockery, V. Hutson, K. Mischaikow and M. Pernarowski, The evolution of slow dispersal rates: A reaction diffusion model, {\it Journal of Mathematical Biology}, {\bf 37} (1998) 61-83.
    
\bibitem{GT24} Q. Ge, D. Tang, Global dynamics of two-species Lotka-Volterra competition-diffusion-advection system with general carrying capacities and intrinsic growth rates, {\it Journal of Dynamics and Differential Equations}, {\bf 36}(2) (2024) 1905-1926.
    
\bibitem{GT23} Q. Ge, D. Tang, Global dynamics of a two-species Lotka-Volterra competition-diffusion-advection system with general carrying capacities and intrinsic growth rates II: Different diffusion and advection rates, {\it Journal of Differential Equations}, {\bf 344} (2023) 735-766.
    
\bibitem{GMK97} S.A.H. Geritz, J.A.J. Metz, E. Kisdi, G. Meszena, Dynamics of adaptation and evolutionary branching, {\it Physical Review Letters}, {\bf 78} (1997) 2024-2027.

\bibitem{G77} B.S. Goh, Global stability in many-species systems, {\it The American Naturalist}, {\bf 111} (1977) 135-143.

\bibitem{HN20} Q. Guo, X. He and W.-M. Ni, On the effects of carrying capacity and intrinsic growth rate on single and multiple species in spatially heterogeneous environments, {\it Journal of Mathematical Biology}, {\bf 81}(2) (2020) 403-433.
    
\bibitem{H78} A. Hastings, Global stability in Lotka-Volterra systems with diffusion, {\it Journal of Mathematical Biology}, {\bf 6} (1978) 163-168.
    
\bibitem{HN16I} X. He, W.-M. Ni, Global dynamics of the Lotka-Volterra competition-diffusion system: diffusion and spatial heterogeneity I, {\it Communications on Pure and Applied Mathematics}, {\bf 69} (2016) 981-1014.

\bibitem{HN16II} X. He, W.-M. Ni, Global dynamics of the Lotka-Volterra competition-diffusion system with equal amount of total resources II, {\it Calculus of Variations and Partial Differential Equations}, {\bf 55}(25) (2016) 20pp.

 \bibitem{HN17} X. He, W.-M. Ni, Global dynamics of the Lotka-Volterra competition-diffusion system with equal amount of total resources III, {\it Calculus of Variations and Partial Differential Equations}, {\bf 56}(132) (2017) 26pp.
     
    
\bibitem{IM18} S.F. Iglesias, S. Mirrahimi, Long time evolutionary dynamics of phenotypically structured populations in time-periodic environments, {\it SIAM Journal on Mathematical Analysis}, {\bf 50}(5) (2018) 5537-5568.    

\bibitem{IM98} M. Iida, T. Muramatsu, H. Ninomiya and E. Yanagida, Diffusion-induced extinction of a superior species in a competition system, {\it Japan Journal of Industrial and Applied Mathematics}, {\bf 15} (1998) 233-252.
    
\bibitem{JR11} P.E. Jabin, G. Raoul, On selection dynamics for competitive interactions, {\it Journal of Mathematical Biology}, {\bf 63}(3) (2011) 493-517.

\bibitem{JS23} P.E. Jabin, R.S. Schram, Selection-Mutation dynamics with spatial dependence, {\it Journal de Math\'{e}matiques Pures et Appliqu\'{e}es}, {\bf 176} (2023) 1-17.
    
\bibitem{J10} A. J\"{u}ngel, Diffusive and non-diffusive population models, {\it In Mathematical Modeling of Collective Behavior in Socio-Economic and Life Sciences}, (2010) 397-425.
    
\bibitem{LN12} K.-Y. Lam, W.-M. Ni, Uniqueness and complete dynamics in the heterogeneous competition-diffusion systems, {\it SIAM Journal on Applied Mathematics}, {\bf 72} (2012) 1695-1712.
        
\bibitem{LC19} T. Lorenzi, A. Marciniak-Czochra, T. Stiehl, A structured population model of clonal selection in acute leukemias with multiple maturation stages, {\it Journal of Mathematical Biology}, {\bf 79}(5) (2019) 1587-1621.
    
\bibitem{L20} T. Lorenzi, C. Pouchol, Asymptotic analysis of selection-mutation models in the presence of multiple fitness peaks, {\it Nonlinearity}, {\bf 33}(11) (2020) 5791-5816.
    
\bibitem{LP21} T. Lorenzi, A. Pugliese, M. Sensi, A. Zardini, Evolutionary dynamics in an SI epidemic model with phenotype-structured susceptible compartment, {\it Journal of Mathematical Biology}, {\bf 83}(6-7) (2021) Paper No. 72.

\bibitem{LVL18} T. Lorenzi, C. Venkataraman, A. Lorz, M.A. Chaplain, The role of spatial variations of abiotic factors in mediating intratumour phenotypic heterogeneity, {\it Journal of Theoretical Biology}, {\bf 451} (2018) 101-110.

\bibitem{LL15} A. Lorz, T. Lorenzi, J. Clairambault, A. Escargueil and B. Perthame, Modeling the effects of space structure and combination therapies on phenotypic heterogeneity and drug resistance in solid tumors, {\it Bulletin of Mathematical Biology}, {\bf 77}(1) (2015) 1-22.

\bibitem{LLH13} A. Lorz, T. Lorenzi, M.E. Hochberg, J. Clairambault and B. Perthame, Populational adaptive evolution, chemotherapeutic resistance and multiple anti-cancer therapies, {\it ESAIM: Mathematical Modelling and Numerical Analysis-Mod\'{e}lisation Math\'{e}matique et Analyse Num\'{e}rique}, {\bf 47}(2) (2013) 377-399.
    
\bibitem{L11} A. Lorz, S. Mirrahimi and B. Perthame, Dirac mass dynamics in multidimensional nonlocal parabolic equations, {\it Communications in Partial Differential Equations}, {\bf 36}(6) (2011), 1071-1098.
    
\bibitem{L06} Y. Lou, On the effects of migration and spatial heterogeneity on single and multiple species, {\it Journal of Differential Equations}, {\bf 223} (2006) 400-426.
    
\bibitem{LN96} Y. Lou, W.-M. Ni, Diffusion, self-diffusion and cross-diffusion, {\it Journal of Differential Equations}, {\bf 131} (1996) 79-131.

\bibitem{L24} N. Loy, B. Perthame, A Hamilton-Jacobi approach to nonlocal kinetic equations, {\it Nonlinearity}, {\bf 37}(10) (2024) Paper No. 105019, 35pp.


\bibitem{P15} B. Perthame, Parabolic Equations in Biology: Growth, reaction, movement and diffusion. Lecture Notes on Mathematical Modelling in the Life Sciences, Springer International Publishing, Paris, 2015.

\bibitem{P06} B. Perthame, Transport Equations in Biology. Birkhaeuser Verlag, Basel-Boston-Berlin, 2007.

\bibitem{PB08} B. Perthame, G. Barles, Dirac concentrations in lotka-volterra parabolic pdes, {\it Indiana University Mathematics Journal}, {\bf 57}(7)  (2008) 3275-3301.

\bibitem{PC18} C. Pouchol, J. Clairambault, A. Lorz and E. Tr\'{e}lat, Selection-Mutation dynamics with spatial dependence, {\it Journal de Math\'{e}matiques Pures et Appliqu\'{e}es}, {\bf 116} (2018) 268-308.
    
\bibitem{PT18} C. Pouchol, E. Tr\'{e}lat, Global stability with selection in integro-differential lotka-volterra systems modelling trait-structured populations, {\it Journal of Mathematical Biology}, {\bf 12}(1) (2018) 872-893.
    
\bibitem{MP73} J.M. Smith, G.R. Price, The logic of animal conflict, {\it Nature}, {\bf 246} (1973) 15-18.

\bibitem{W25} Q. Wang, The effects of diffusion coefficients in a two-species Lotka-Volterra competition system with resource dependent dispersal, {\it Acta Applicandae Mathematicae}, {\bf 195}(8) (2025) 20pp.
    
\bibitem{Z82} L. Zhou, C.V. Pao, Asymptotic behavior of a competition—diffusion system in population dynamics, {\it Nonlinear Analysis}, {\bf 6} (1982) 1163-1184. 

\end{thebibliography}
\end{document}